\documentclass{article}

\PassOptionsToPackage{hyphens}{url}\usepackage[backref,colorlinks,citecolor=blue,bookmarks=true]{hyperref}

\usepackage[utf8]{inputenc}
\usepackage{style-det}
\usepackage{authblk}

\usepackage{enumitem}

\newtheorem{main}{Theorem}

\newcommand{\sub}{\subseteq}

\newcommand{\floor}[1]{\lfloor{#1}\rfloor}
\newcommand{\ceil}[1]{\lceil{#1}\rceil}

\DeclareMathOperator{\Sp}{Span}
\renewcommand{\det}{\mathrm{det}}
\newcommand{\e}{\epsilon}
\newcommand{\sm}{\setminus}

\title{Slice rank and partition rank of the determinant}

\date{}

\author[1]{Amichai Lampert}
\author[2]{Guy Moshkovitz}
\affil[1]{University of Michigan}
\affil[2]{City University of New York}

\begin{document}
 \maketitle

\begin{abstract}
    The Laplace expansion expresses the $n \times n$ determinant $\det_n$ as a sum of $n$ products. Do shorter expansions exist? In this paper we:
    \begin{itemize}
	\item Fully determine the slice rank decompositions of $\det_n$ (where each product must contain a linear factor): In this case, we show that $n$ summands are necessary, and moreover, the only such expansions with $n$ summands are equivalent (in a precise sense) to the Laplace expansion.	
	\item Prove a logarithmic lower bound for the partition rank of $\det_n$ (where each product is of multilinear forms): In this case, we show that at least $\log_2(n)+1$ summands are needed. We also explain why existing techniques fail to yield any nontrivial lower bound, and why our new method cannot give a super-logarithmic lower bound.
	\item Separate partition rank from slice rank for $\det_n$: we find a quadratic expansion for $\det_4$, over any field, with fewer summands than the Laplace expansion. This construction is related to a well-known example of Green-Tao and Lovett-Meshulam-Samorodnitsky disproving the naive version of the Gowers Inverse conjecture over small fields.
\end{itemize}
    An important motivation for these questions comes from the challenge of separating structure and randomness for tensors. On the one hand, we show that the random construction fails to separate: for a random tensor of partition rank $r$, the analytic rank is $r-o(1)$ with high probability. On the other hand, our results imply that the determinant yields the first asymptotic separation between partition rank and analytic rank of $d$-tensors, with their ratio tending to infinity with $d$.
\end{abstract}

\section{Introduction}

The determinant is one of the most well known polynomials in mathematics, and formulas for it are naturally important. The Laplace (row) expansion is the principal such formula, expressing the determinant of $n \times n$ matrices as a sum of products:
for any $i$, $\det(A) = \sum_{j=1}^n (-1)^{i+j}a_{i,j} \det(A_{i,j}),$
where $A_{i,j}$ is obtained from $A$ by removing row $i$ and column $j$.
The number of summands in the Laplace expansion for the $n \times n$ determinant is $n$. Are there expansions of the determinant with fewer summands?
In this paper we initiate the study of such expansions. This follows in the tradition of~\cite{Derksen16,Land-Teit10}, where other notions of rank for the determinant were studied.


We denote by $\det_n$ the determinant polynomial of $n \times n$ matrices. The \emph{slice rank}~\cite{Tao16,SawinTao16} of the determinant equals the minimal number of summands when we require each of the products to have a linear factor. Our first result completely characterizes these expansions.

\begin{main}[see Theorem~\ref{thm:slice-unique}] 
    $\srk(\det_n)=n$ for every $n\ge 2$, and the only minimum slice rank expansions are all ``equivalent'' to the Laplace expansion.
\end{main}

Since the determinant is a \emph{multilinear form}---meaning a homogeneous polynomial that is linear in each of its vector variables\footnote{In the case of the determinant, the rows of the matrix (or, alternatively, the columns).}---it is natural to require that the terms of the expansion are multilinear as well. For a multilinear form $T$, 
the \emph{partition rank}~\cite{Naslund20} $\prk(T)$ is the least number of summands when we require each of the factors to be multilinear. Our results for partition rank of the determinant are as follows:

\begin{main}[see Theorem~\ref{thm:prk-LB}]\label{main:LB}
    For every $n\ge 2$, $\prk(\det_n) \ge \log_2(n)+1$.
\end{main}

\begin{main}[see Theorem~\ref{prop:prk-4}]\label{main:prk-4}
    It is \emph{not} true that $\prk(\det_n)=n$, or that $\prk(\det_n)=\srk(\det_n)$.
    Namely, $\prk(\det_4)=3$ over any field.
\end{main}

We mention that the partition rank expansion of $\det_4$ in Theorem~\ref{main:prk-4} is genuinely different from the Laplace expansion or the generalized Laplace expansion (see Lemma~\ref{not-Laplace}).


Note that we consider expressions of the determinant as, simply, a sum of products, each with at least two factors.
For comparison, in (multilinear) arithmetic formulas or circuits, each factor in the expansion is also given an expansion, as well as each factor in those expansions and so on.
In this context, what we are interested in is only the top fan-in, which is trivially at most $n$.\footnote{An extreme case is where \emph{all} factors in the expansions are required to be linear (this corresponds to the traditional notion of tensor rank), for which an exponential lower bound of $\binom{n}{\floor{n/2}}$ is known for the number of summands~\cite{Derksen16}.}


Our proof strategy for Theorem~\ref{main:LB} (and similar proof strategies used to bound partition rank in the past) cannot yield a lower bound on $\prk(\det_n)$ better than logarithmic (see Remark~\ref{remark:strategy}).
It remains open whether a different strategy could do better.

\begin{problem}
    Determine the asymptotics of $\prk(\det_n)$.
\end{problem}

One can also ask about expansions of the determinant where the factors need not be multilinear. This corresponds to the \emph{strength} (or \emph{rank}) of $\det_n$, where for any form $P$, $\str(P)$ is the least number of reducible forms that sum to $P$.
For this notion, even obtaining \emph{any} nontrivial bound is open.

\begin{conjecture}
    $\lim_{n \to \infty} \str(\det_n) = \infty$.
\end{conjecture}

As far as we know, it could be true that $\str(\det_n) = n$ for all $n \neq 4$!


\subsection{Structure versus randomness}

Proving lower bounds for $\prk(\det_n)$ also has an important application in the study of the structure-versus-randomness phenomenon 
for (multilinear) polynomials.
Structure is usually captured by partition rank, 
and randomness by analytic rank.
For a multilinear form $T$ over a finite field $\F$, denote $\bias(T) = \Pr[T=0]-\Pr[T=y]$ for any nonzero $y \in \F$.\footnote{The definition is independent of the choice of $y\neq 0$.}\footnote{Throughout, the probabilities are over a uniformly random choice in the relevant domain.}
The analytic rank~\cite{GowersWo11} of $T$ is $\ark(T) = -\log_{|\F|}(\bias(T))$.


Sub-additivity~\cite{KazhdanZi18,Lovett19} implies $\ark(T) \le \prk(T)$.
A central conjecture in additive combinatorics (see e.g.~\cite{Lovett19,AdiprasitoKaZi21,Janzer19,Milicevic19})---the \emph{partition-versus-analytic rank conjecture}---posits that the inverse approximately holds: $\prk(T) \le O_d(\ark(T))$ 
(that is, uniformly in $n$ and $\F$).
This conjecture is known to hold if the field is large enough~\cite{CohenMo21B,MoshkovitzZh22}.\footnote{$|\F| \ge (\ark(T)+1)^{\e(d)}$ would do. It is also known to hold up to a logarithmic factor in $\ark(T)$ over any finite field~\cite{MoshkovitzZh22}.}
The extent to which the bound must depend on $d$, however, is unclear, 
and the best lower bound we are aware of is a $3/2$ factor~\cite{CohenMo21} (see also~\cite{KoppartyMoZu20,ChenYe24} for the slightly weaker lower bound of $4/3$ witnessed by the matrix multiplication tensor).

Formally, define the extremal ratio $A(d) = \max_T \prk(T)/\lceil \ark(T) \rceil$, where the maximum is over all $d$-linear forms $T$.\footnote{One can define a version of $A(d)$ over large enough fields, which is necessarily finite by the aforementioned~\cite{CohenMo21B} (even if the partition-versus-analytic rank conjecture turned out to be false). The lower bound on $A(d)$ we prove in this paper indeed holds over any finite field.}\footnote{The ceiling avoids the rather degenerate case where $\prk(T)=1$ but the denominator $\ark(T)$, which need not be an integer, is close to $0$ (when $T$ is a product of many factors).}
Linear algebra tell us that $A(2)=1$ (that is, $\prk(T)=\ark(T)$ for any linear map/bilinear form $T$), 
and as mentioned above, it is known that $A(3) \ge \frac32$. (Note also that $A(d)$ is monotone increasing).

It seems that separating partition rank from analytic 
rank---that is,
proving the existence of a multilinear form that has high partition rank but low analytic rank---is a nontrivial challenge. 
In particular, for almost all $d$-linear forms $T \colon (\F^n)^d \to \F$, both $\prk(T)$ and $\ark(T)$ are essentially $n$. 
However, what if we condition $T$ to have a given partition rank?
We show that one cannot obtain a separation in this way either.
\begin{main}[see Theorem~\ref{thm:random-ark}]
    A random $d$-linear form of partition rank $r$ has, with high probability, analytic rank roughly $r$.
\end{main}

In the context of separating structure from randomness, the importance of obtaining lower bounds for $\prk(\det_n)$ stems from the observation (see Corollary~\ref{coro:ark-det} below) that $\lceil \ark(\det_n) \rceil =2$ for any $n \ge 2$.
Therefore, any lower bound on $\prk(\det_n)$ would yield a corresponding lower bound on $A(d)$, namely $A(d) \ge \prk(\det_d)/2$. 
We therefore obtain from Theorem~\ref{main:LB} the following separation.
\begin{corollary}[See Corollary~\ref{coro:Cd}]
    $A(d) \ge (\log_2(d)+1)/2$, 
    and in particular, $A(d)$ is unbounded.
\end{corollary}

It follows that a dependence on $d$ in the structure-versus-randomness question is inevitable. Tensors whose partition and analytic rank are far apart are provably rare, by Theorem~\ref{thm:random-ark}.
Explicitly, we now know of one such example, the determinant, by Theorem~\ref{thm:prk-LB}. It would be interesting to find additional explicit constructions of such tensors.%

Let us remark that randomness can also be defined geometrically, over any field $\F$, using \emph{geometric rank}~\cite{KoppartyMoZu20}. 
For a $d$-linear form $T=T(x^{(1)},\ldots,x^{(d)})$, denote by $\nabla T(x^{(1)},\ldots,x^{(d-1)})=(\partial T/\partial x^{(d)}_i)_i$ the gradient of $T$ with respect to the last (say) vector variable.
Over finite fields, the gradient is intimately connected to the analytic rank; indeed, a simple Fourier-analytic argument (e.g.~\cite{Lovett19}) shows that $\bias(T) = \Pr[\nabla T = 0]$.
This motivates the definition of geometric rank, which is the geometric analogue of analytic rank. Formally, $\grk(T) = \codim\{x \vert \nabla T(x)=0\}$, the codimension of the algebraic variety (over the algebraic closure of $\F$) cut out by the gradient.
It is known that $\grk(T)=\Theta_d(\ark(T))$ over any finite field~\cite{ChenYe24,MoshkovitzZh24},
and that $\grk(T)=\Theta_d(\prk(T))$ over any algebraically closed field~\cite{KazhdanLaPo24,Schmidt85}.
We can now define a \emph{geometric} analogue to $A(d)$, namely 
$G(d) = \max_T \prk(T)/\grk(T)$, where the maximum is over all $d$-linear forms.
One can verify that all our results about separating structure and randomness translate, with some effort, to the geometric setting, in which analytic rank is replaced by geometric rank.
In particular, $G(d) \ge \frac{\log_2(d)+1}{2}.$


It remains open to obtain a super-logarithmic separation between structure and randomness, either using an explicit construction or not, and either in the finite field setting or in the geometric setting.
\begin{problem}
    Prove or disprove: $A(d) =\omega(\log d)$. Similarly for $G(d)$.
\end{problem}

\subsection{Acknowledgments} AL was supported by NSF Award DMS-2402041. He is also grateful to the City University of New York for its hospitality during the period when this project was initiated. 
GM was supported by NSF Award DMS-2302988 and a PSC-CUNY award.

\section{Slice rank of the determinant}

\renewcommand{\a}{\vec{\alpha}}
\newcommand{\g}{\vec{g}}

We begin this section with the definition of slice rank:

\begin{definition}
    If $f$ is a form (i.e. homogeneous polynomial) of degree $d$ then its \emph{slice rank}, written $\srk(f),$ is the minimal $r$ such that $f = \sum_{i=1}^r \alpha_i g_i,$ where $\alpha_i$ are linear forms and $g_i$ are forms of degree $d-1.$\footnote{Slice rank is usually defined for multilinear forms and requires the decomposition to respect the multilinear structure. Our less restrictive definition means that our characterization of slice rank decompositions of the determinant holds in greater generality.}
\end{definition}
In this section we will characterize the slice rank decompositions of the determinant.  We will prove that the slice rank of $\det_n$ is $n$, and show that any slice rank decomposition of length $n$ is equivalent to the Laplace expansion, in a sense we will soon make precise.

\subsection{Equivalence for slice rank decompositions}  

We now turn  to describing some natural reductions for slice rank decompositions. See Karam~\cite{Karam24} for a more detailed discussion. Let $f$ be a form of degree $d$, and let $f = \sum_{i=1}^r \alpha_i g_i$ be some (initial) slice rank decomposition. We will denote this decomposition by $(\vec{\alpha}, \vec{g}).$

\begin{definition}\label{reductions}
    We define a preorder $\lesssim$ on the collection of slice rank decompositions of $f$ by taking the transitive closure of the following rules:
    \begin{enumerate}
        \item \textbf{Linear reduction:} Suppose that $\beta_1,\ldots,\beta_r$ are linear forms such that $\alpha_1,\ldots,\alpha_r\in \Sp(\beta_1,\ldots,\beta_r) .$ 

Writing  $\alpha_i  = \sum_{j=1}^r c_{i,j}\beta_j,$ we get:
\[
f = \sum_{i=1}^r \alpha_i g_i =  \sum_{i=1}^r \sum_{j=1}^r c_{i,j}\beta_j g_i = \sum_{j=1}^r \beta_j \sum_{i=1}^r c_{i,j} g_i.  
\]
Thus, $(\vec{\beta},\vec{h})$ is another slice rank decomposition, where $h_j = \sum_{i=1}^r c_{i,j}g_i.$ We declare $(\vec{\beta},\vec{h}) \lesssim (\vec{\alpha}, \vec{g}).$
\item \textbf{Syzygy reduction:} If $h_i = g_i-\sum_{j\neq i} q_{i,j}\alpha_j$ for an alternating matrix  $(q_{i,j})_{i,j}$ of forms of degree $d-2$ (i.e. $q_{i,i} = 0$ and $q_{i,j} = -q_{j,i}$), then
\[
f = \sum_{i=1}^r \alpha_i g_i  = \sum_{i=1}^r \alpha_i h_i + \sum_{1\le i<j \le r} \alpha_i\alpha_j (q_{i,j}+q_{j,i}) = \sum_{i=1}^r \alpha_i h_i. 
\]
Therefore, $(\vec{\alpha}, \vec{h})$ is another slice rank decomposition with $(\vec{\alpha}, \vec{h}) \lesssim (\vec{\alpha}, \vec{g}).$ 
\item \textbf{Symmetric reduction:} If $A$ is a linear map such that  $f\circ A = f$,\footnote{For example, if $f=\det_n$ then matrix transposition is one example of such a linear map (that is, $f$ is an invariant under matrix transposition).} then $(\vec{\alpha}\circ A, \vec{g}\circ A)$ is also a slice rank decomposition with
$(\vec{\alpha}\circ A, \vec{g}\circ A)\lesssim (\vec{\alpha}, \vec{g}).$
    \end{enumerate}
\end{definition}

\begin{definition}\label{srk-eq}
    Two slice rank decompositions $(\vec{\alpha}, \vec{g})$ and $(\vec{\beta}, \vec{h})$ are \emph{equivalent} if both $(\vec{\beta}, \vec{h}) \lesssim (\vec{\alpha}, \vec{g})$ and $(\vec{\alpha}, \vec{g})\lesssim (\vec{\beta}, \vec{h})$.
\end{definition}

Our goal here is to show that, up to our equivalence defined above, slice rank decompositions are determined by their (linearly independent) linear factors.

\begin{lemma}\label{lem:srk-span}
    Suppose that $(\vec{\alpha}, \vec{g})$ and $(\vec{\beta}, \vec{h})$ are two slice rank decompositions of $f.$ If $\beta_1,\ldots,\beta_r$ are linearly independent and their linear span contains $\alpha_1,\ldots,\alpha_r,$ then $(\vec{\beta}, \vec{h})\lesssim (\vec{\alpha}, \vec{g}).$  
\end{lemma}

\begin{proof}
    By assumption, $\alpha_i = \sum_j c_{i,j}\beta_j.$ Plugging this in yields
    \[
    0 = f-f = \sum_{j=1}^r \beta_j (h_j - \sum_{i=1}^r c_{i,j} g_i). 
    \]
    The proof is complete once we show that if $t_1,\ldots,t_r$ are forms of degree $d-1$  satisfying $0 = \sum_{j=1}^r \beta_j t_j$ then $t_j = \sum_{i\neq j} q_{i,j}\beta_i$ for forms $q_{i,j}$ of degree $d-2$ which satisfy  $q_{i,j} = -q_{j,i}$. This we prove by induction on $r,$ following an argument from~\cite{usr0192}.
    
   In the base case $r=1$ the equality $\beta_1 t_1 = 0$ implies $t_1 = 0$ as claimed. Suppose the claim holds for $r-1$ and that $0 = \sum_{j=1}^r \beta_j t_j.$ Restricting to the hyperplane $W$ where $\beta_r$ vanishes, we have $ 0 = \left( \sum_{j=1}^{r-1} \beta_j t_j \right) \restriction_W.$ The linear forms $\beta_1,\ldots,\beta_{r-1}$ are linearly independent when restricted to $W,$ so the inductive hypothesis implies  
   that for $j<r$ we have $t_j\restriction_W = \sum_{i\neq j,r} q_{i,j}\beta_i \restriction_W$, where $q_{i,j}$ are forms on $W$ of degree $d-2$ satisfying $q_{i,j} = -q_{j,i}$ for all $i<j<r$. 
   
   We lift $q_{i,j}$ to forms of degree $d-2$ on the whole space while preserving the skew-symmetry $q_{i,j} = -q_{j,i}$ for all $i<j<r.$   
   This gives us $t_j = \sum_{i\neq j,r} q_{i,j}\beta_i +s_j\beta_r $ for $j< r$, where $s_j$ are forms of degree $d-2.$ Therefore,
   \[
   0 = \sum_{j=1}^r \beta_j t_j = \sum_{j=1}^{r-1} \beta_j \left(\sum_{i\neq j,r}  q_{i,j}\beta_i + s_j\beta_r \right) + \beta_r t_r = \beta_r \left(\sum_{j=1}^{r-1} s_j\beta_j +t_r \right).
   \]
   Dividing both sides by $\beta_r$ yields  $t_r = -\sum_{i=1}^{r-1} s_i\beta_i,$ proving the inductive claim with $q_{r,i} = s_i =  -q_{i,r}$ for $i<r$.  
 
\end{proof}

\begin{corollary}\label{srk-equiv}
    If $(\vec{\alpha}, \vec{g})$ and $(\vec{\beta}, \vec{h})$ are two slice rank decompositions of $f$, and $\alpha_i$ and $\beta_i$ are bases for the same linear space, then the two slice rank decompositions are equivalent.
\end{corollary}

\subsection{Slice rank decompositions of the determinant}

Our goal is to prove the following:

\begin{theorem}\label{thm:slice-unique}
    For every $n\ge 2$, the slice rank of $\det_n$ equals $n.$ Moreover, any slice rank decomposition $\det_n = \sum_{i=1}^n \alpha_i g_i$ is equivalent (in the sense of Definition~\ref{srk-eq}) to the Laplace expansion along the first row.
\end{theorem}

Let $M_n(\F)$ denote the space of $n\times n$ matrices with coefficients in $\F.$ Our main tool is the following result of Meshulam~\cite{Meshulam85}. We are interested specifically in the case $s = n-1.$

\begin{proposition}[Meshulam]\label{thm:Meshulam}
Let $\mathcal{M}\subseteq M_n(\F)$ be a subspace.
\begin{enumerate}
    \item If $\dim(\mathcal{M}) > sn$ then it contains a matrix of rank greater than $s.$ 
    \item If $\dim(\mathcal{M}) = sn$ and $\rk(x) \le s$ for all $x\in \mathcal{M},$ then there exists a subspace $E\subset \F^n$ of dimension $s$ such that $\mathcal{M} = E\otimes \F^n$ or $\mathcal{M}^t = E\otimes \F^n,$ where 
    \[
     E\otimes \F^n:= \Sp\{uv^t: u\in E,v\in \F^n\}.
    \]
\end{enumerate} 
\end{proposition}

We introduce the notation $R_1 = \Sp(x_{1,1},\ldots,x_{1,n})$ for the linear space spanned by the coordinates of the first row.

\begin{corollary}\label{coro:Meshulam}
    Suppose $\det_n = \sum_{i=1}^r \alpha_i g_i$ is a slice rank decomposition. Then:
    \begin{enumerate}
        \item The inequality $r\ge n$ holds and 
        \item If $r =n$ there exists $A\in \textnormal{SL}_n(\F)$ such that either $(\alpha_i(Ax))_{i\in [n]}$ or $(\alpha_i(Ax^t))_{i\in [n]}$ is a basis for $R_1.$
    \end{enumerate} 
\end{corollary}

\begin{proof}
    Let
    \[
    \mathcal{M} = \{x\in M_n(\F) : \alpha_1(x) = \ldots \alpha_r (x) = 0\}.
    \]
    By assumption, $\det_n$ vanishes identically on $\mathcal{M},$ so all the matrices in $\mathcal{M}$ have rank at most $n-1.$ By Proposition~\ref{thm:Meshulam},
    \[
    (n-1)n \ge \dim \mathcal{M} \ge n^2-r \implies r\ge n.
    \] 
    This proves the first assertion. 
    Moreover, if $r = n$ then we are in the second case of Proposition~\ref{thm:Meshulam} (with $s=n-1$), so $\mathcal{M} = \F^n\otimes E$ or $\mathcal{M} = E\otimes \F^n$ for some $(n-1)$-dimensional subspace $E \subset \F^n$.
    (Equivalently, either all matrices in $\mathcal{M}$ have row space $E$, or they all have column space $E$.)
    We may then choose $A\in \textnormal{SL}_d(\F)$ 
    such that either $A\mathcal{M}$ or $A\mathcal{M}^t$ is the space of matrices with zeros in the first row, which means that $(\alpha_i(A^{-1}x))_{i\in [n]}$ or $(\alpha_i(A^{-1}x^t))_{i\in [n]}$ is a basis for $R_1,$ respectively. Indeed, let $A$ map $E$ to the subspace $E' \sub \F^n$ spanned by the standard basis vectors $e_2,\ldots,e_n$, 
    and note that $A\mathcal{M} = A(E\otimes \F^n) = \Sp\{Au v^t: u\in E,v\in \F^n\} = E' \otimes \F^n$, as desired.  
\end{proof}

\begin{proof}[Proof of Theorem~\ref{thm:slice-unique}]

We have already proved that the slice rank of $\det_n$ equals $n.$ 
Let $\det_n = \sum_{i=1}^n \alpha_i g_i$ is a slice rank decomposition with $n$ summands. 
By Corollary~\ref{coro:Meshulam}, there exists $A\in \textnormal{SL}_n(\F)$ such that either $\alpha_i(Ax)$ or $\alpha_i(Ax^t)$ is a basis for $R_1.$ Note that the determinant is invariant under the invertible linear maps $x\mapsto x^t$ and $x\mapsto Ax$ (for any $A \in \textnormal{SL}_n(\F)$), so this constitutes a symmetric reduction. 
Let $(\vec{\alpha'},\vec{g})$, with $\alpha'_i = \alpha_i \circ A$ for every $i$, be the corresponding decomposition of $\det_n$, which is thus equivalent to our initial decomposition.
Now, if $(\vec{\beta}, \vec{h})$ denotes the Laplace expansion along the first row, then
we have that both $\vec{\alpha'}$ and $\vec{\beta}$ are bases for $R_1$.
Applying Corollary~\ref{srk-equiv}, we deduce that they are equivalent, which, by transitivity, completes the proof.
\end{proof}

\section{Partition rank of the determinant}

In this section we prove nontrivial lower and upper bounds on $\prk(\det_n)$, the partition rank of the determinant, when viewed as a multilinear function in the rows of a matrix. We explain why current techniques fail to obtain any nontrivial lower bound and also why our lower bound cannot be improved even with our new argument.


\subsection{Partition rank lower bound}

Our proof of the lower bound will proceed by repeatedly finding an assignment that zeros out at least one term in the expansion. 
It crucially uses the highly symmetric nature of the determinant. We will need the following simple lemma.

\begin{lemma}\label{lemma:ind-zero}
    Let $\vec{Q}\colon(\F^n)^k \to \F^\ell$ be a $k$-linear map with $k+\ell \le n$.
    Then there exist linearly independent $v_1,\ldots,v_k \in \F^n$ such that $\vec{Q}(v_1,\ldots,v_k)=0$.
\end{lemma}
\begin{proof}
    Let $v_1,\ldots,v_{k-1} \in \F^n$ be any $k-1$ linearly independent vectors.
    Consider the restriction $L(x)=\vec{Q}(v_1,\ldots,v_{k-1},x)$, which is a linear map
    $L \colon \F^n \to \F^\ell$.
    We have $\dim\ker(L) \ge n-\ell \ge k$, by assumption.
    Thus, there exists $v_k \in \ker(L) \setminus \Sp\{v_1,\ldots,v_{k-1}\}$.
    The vectors $v_1,\ldots,v_k$ are linearly independent and satisfy $\vec{Q}(v_1,\ldots,v_k)=L(v_k)=0$, as needed.
\end{proof}

We introduce one more piece of useful notation. If $T \colon (\F^n)^d\to \F$ is a multilinear form, $k<d$, and $v_1,\ldots,v_k\in \F^n$ are fixed, then $T[v_1,\ldots,v_k]:(\F^n)^{d-k}\to \F$ is the multilinear form obtained by fixing the first $k$ vectors in $T.$ We can now proceed with our proof of the lower bound.
In everything that follows, $\lg(x)$ stands for $\log_2(x)$.

\begin{theorem}\label{thm:prk-LB}
    $\prk(\det_n) \ge \lg(n)+1$.
\end{theorem}

\begin{proof}
    This follows by induction on $n$ from the statement that there exists some $n/2 \le m <n$ with $\prk(\det_n) > \prk(\det_m).$ Indeed,
    $$\prk(\det_n) > \prk(\det_m) \ge \lceil\lg(m)\rceil+1\ge \lceil\lg(n/2)\rceil+1 \ge \lceil\lg(n)\rceil.$$

    To prove this statement, consider a partition rank decomposition of $\det_n$ of minimal length $\det_n= \sum_{i=1}^{r} Q_iR_i,$
    where for each $i$, $Q_i$ and $R_i$ are multilinear forms in disjoint sets of variable vectors.
    Let $I_i \subset [n]$ correspond to the subset of vector variables that $Q_i \in (\F^n)^{|I_i|} \to \F$ depends on (so $|I_i|=\deg(Q_i)$, and $R_i$ depends on $I_i^c := [n]\sm I_i$).
    We choose our notation so that 
    $|I_i| \le n/2$ for all $i\in [r]$, and $1\in I_i$ if $|I_i| = n/2$ (as a tiebreaker).
    
    Let $I$ be a minimal set in $\{I_i\}_{i\in [r]}$, 
    and denote $k=|I| \le n/2.$ Assume without loss of generality that $I = [k],$ applying a permutation to the rows if necessary. Suppose that $I_1 = \cdots = I_\ell = [k],$ and $I_i\neq [k]$ for all $i>\ell.$ We may assume that $k+r \le n$ since otherwise we are done by
    \[
    \prk(\det_n) = r > n-k \ge \prk(\det_{n-k}).
    \]

    Applying lemma \ref{lemma:ind-zero} to the multilinear map $(Q_1,\ldots,Q_\ell) \colon (\F^n)^k \to \F^\ell,$
     we deduce that there exist linearly independent $v_1,\ldots,v_k$ with $Q_i(v_1,\ldots,v_k) = 0$ for all $i\in [\ell].$ Let $A\in \textnormal{SL}_n(\F)$ be a matrix with $Ae_i = v_i$ for $i=1,\ldots,k,$ where $e_i$ are the standard basis vectors. Note that
    \[
    \det_n (x_1,\ldots,x_n) = \det_n(Ax_1,\ldots,Ax_n) = \sum_{i=1}^r (Q_i\circ A) \cdot (R_i\circ A)
    \]
    so that we may assume without loss of generality that $v_i = e_i$ for $i=1,\ldots,k.$ 
     
    By minimality of $I$ and our convention for choosing the $I_i,$ we get that for any $i>\ell,$ both $I_i \nsubseteq [k]$ and also $[n]\setminus I_i \nsubseteq [k].$ 
    Therefore, 
    \[\det_{n-k}=\det_n[e_1,\ldots,e_k] = \sum_{i=\ell+1}^r (Q_iR_i)[e_1,\ldots,e_k]\] 
    is a partition rank decomposition of length $r-\ell.$ 
    Thus,
    \[
    \prk(\det_{n-k}) \le r-\ell < r = \prk(\det_n).
    \]
    This completes the proof of the theorem, recalling that $n-k \ge n/2.$ 
    
\end{proof}

\begin{remark}\label{remark:strategy}
    Existing techniques for lower-bounding the  partition rank or the slice rank of tensors~\cite{Tao16,Naslund20} proceed by showing that $\prk(T)\ge \prk(T[u])$ for a single well-chosen $u\in \F^n.$ For the determinant this inductively gives the trivial inequality $\prk(\det_n) \ge \prk(\det_2) = 2.$ Another known approach is via the inequality $\prk(T) \ge \grk(T)$~\cite{KoppartyMoZu20}. In our case $\grk(\det_n) = 2$ for all $n\ge 2$, so this again yields the trivial bound $ \prk(\det_n)\ge 2.$
    
     On the other hand, our strategy involves (carefully) fixing many---potentially $n/2$---of the $n$ vector variables at each step. We note that this novel proof strategy can never yield a super-logarithmic lower bound. This is because such an assignment might zero out only a single summand in the expansion, at the the expense of reducing $\det_n$ to $\det_{n/2}.$ Thus, repeating such an argument will always result in a partition rank lower bound of roughly $\lg(n)$. 
\end{remark}

\begin{remark}\label{remark:small-prk}
    For small values of $n$, the bound $\prk(\det_n) \ge \lceil\lg(n)\rceil+1$ in Theorem~\ref{thm:prk-LB} implies $\prk(\det_2) = 2$, $\prk(\det_3) = 3$, and $\prk(\det_4) \ge 3$.
\end{remark}

Next, we show that our lower bound is tight for $\det_4$.

\subsection{Partition rank upper bound}

Here we demonstrate that, unlike for slice rank, the partition rank of the determinant is not always full. We do so by presenting a (perhaps surprising) ``quadratic'' expansion of the determinant of $4 \times 4$ matrices, quite unlike the Laplace expansion.


For an $n \times n$ matrix $A$ and subsets $I,J \sub [n]$, 
we denote by $A_{I,J}$ the $|I| \times |J|$ submatrix of $A$ on the rows indexed by $I$ and the columns indexed by $J$.\footnote{Put differently, $A_{I,J}$ removes from $A$ all rows in $[n]\sm I$ and all columns in $[n]\sm J$.}

\begin{theorem}\label{prop:prk-4}
    For any $4 \times 4$ matrix $A$ over any field,
    \begin{align}\label{eq:d4-expansion}
    \begin{split}
        \det(A) &= 
            \bigg(\sum_J \det(A_{\{1,2\},J}) \bigg)\bigg(\sum_J \det(A_{\{3,4\},J}) \bigg)\\
            &- \bigg(\sum_J \det(A_{\{1,3\},J}) \bigg)\bigg(\sum_J \det(A_{\{2,4\},J}) \bigg)\\
            &+ \bigg(\sum_J \det(A_{\{1,4\},J}) \bigg)\bigg(\sum_J \det(A_{\{2,3\},J}) \bigg) ,
    \end{split}
    \end{align}
    where all summations are over $J \in \binom{[4]}{2}$ (all choices of $2$ out of the $4$ columns).
\end{theorem}
\begin{corollary}\label{coro:det4}
    Over any field,
    $\prk(\det_4) = 3$.
\end{corollary}

This result is partially inspired by the results of Green-Tao~\cite{GreenTao09} and Lovett-Meshulam-Samorodnitsky~\cite{LovettMeSa11}, who showed that $S_4$, the elementary symmetric polynomial of degree $4$, is a counterexample to the naive version of the inverse conjecture for Gowers norms in small characteristic, namely over $\F_2$.
The connection to our question comes from the fact that over $\F_2,$ the multilinear form corresponding to $S_4$ in four variables is precisely the  $4 \times 4$ determinant.

As a quick example for computing the determinant using~(\ref{eq:d4-expansion}), consider the simple case of diagonal matrices:
\[
\det\begin{pmatrix}
    a & 0 & 0 & 0\\
    0 & b & 0 & 0\\
    0 & 0 & c & 0\\
    0 & 0 & 0 & d
\end{pmatrix}
= (ab)(cd) - (ac)(bd) + (ad)(bc).
\]

For the interested reader, we show below in Section~\ref{subsubsec:exp-comp} that our expansion~(\ref{eq:d4-expansion}) is genuinely different from the Laplace expansion and the generalized Laplace expansions.

At the core of our upper bound 
is an identity relating four-dimensional and two-dimensional Levi-Civita symbols.
The $d$-dimensional Levi-Civita symbol, 
for $i_1,\ldots,i_d \in [d]$, is
\[\e_{i_1,\ldots,i_d} = 
    \begin{cases}
        1 & (i_1,\ldots,i_d) \text{ is a permutation of even sign}\\
        -1 & (i_1,\ldots,i_d) \text{ is a permutation of odd sign}\\
        0 & \text{otherwise}
    \end{cases}.\]
Here, $(i_1,\ldots,i_d)$ is the (not necessarily permutation) map $1 \mapsto i_1, \ldots, d \mapsto i_d$.
The Levi-Civita symbol naturally generalizes to arbitrary integers $i_1,\ldots,i_d$, which we will use in the $2$-dimensional case; namely, for any $i,j$, 
\[\e_{i,j} = 
    \begin{cases}
        1 & i<j\\
        -1 & i>j\\
        0 & i=j
    \end{cases}.\]

\begin{proof}[Proof of Theorem~\ref{prop:prk-4}]
    First, we write out each determinant appearing in~(\ref{eq:d4-expansion}):
    \begin{align*}
    \begin{split}
        \det(A) &= 
            \Big(\sum_{i,j} \e_{i,j}a_{1,i}a_{2,j}\Big)\Big(\sum_{i,j} \e_{i,j}a_{3,i} a_{4,j} \Big)\\
            &- \Big(\sum_{i,j} \e_{i,j}a_{1,i}a_{3,j}\Big)\Big(\sum_{i,j} \e_{i,j}a_{2,i}a_{4,j} \Big)\\
            &+ \Big(\sum_{i,j} \e_{i,j}a_{1,i}a_{4,j}\Big)\Big(\sum_{i,j} \e_{i,j}a_{2,i}a_{3,j} \Big) ,
    \end{split}
    \end{align*}
    where all summations are over $i,j \in [4]$ (in fact, can assume $i\neq j$ as $\e_{i,i}=0$).
    Comparing the coefficients of $a_{1,i}a_{2,j}a_{3,k}a_{4,\ell}$ in both sides for each $(i,j,k,\ell)$, our expansion turns out to be equivalent to the following ``$4$-to-$2$'' identity of Levi-Civita symbols: for all $i,j,k,w \in [4]$,
    \begin{equation}\label{4-to-2}
        \e_{i,j,k,\ell} = \e_{i,j}\e_{k,\ell} - \e_{i,k}\e_{j,\ell} + \e_{i,\ell}\e_{j,k}.
    \end{equation}
    \[\]
    A direct inspection reveals that if the identity holds for $(i,j,k,\ell)$, then it also holds after transposing any two indices. For example,
    \[
    \e_{j,i,k,\ell} = -\e_{i,j,k,\ell} = -(\e_{i,j}\e_{k,\ell} - \e_{i,k}\e_{j,\ell} + \e_{i,\ell}\e_{j,k})
    = \e_{j,i}\e_{k,\ell} - \e_{j,k}\e_{i,\ell} + \e_{j,\ell}\e_{i,k}.
    \]
    Since all permutations are obtained by a sequence of transpositions, we may apply any desired permutation to $i,j,k,\ell$ in the course of our proof.
    
    If $i,j,k,\ell$ are not all distinct then by the above property we may assume $i=j$. Then the left hand side of equation \eqref{4-to-2} is $\e_{i,i,k,\ell}=0$, since $(i,i,k,\ell)$ is not a permutation, and the right hand side is $-\e_{i,k}\e_{i,\ell} + \e_{i,\ell}\e_{i,k} = 0$ also.
    
    If $i,j,k,\ell$ are all distinct, then because transpositions generate the group of permutations, we may assume $(i,j,k,\ell) = (1,2,3,4).$ In this case,
    the left hand side of equation \eqref{4-to-2} equals $1$,
    and the right hand side equals
    \[
    \e_{1,2}\e_{3,4} - \e_{1,3}\e_{2,4} + \e_{1,4}\e_{2,3}=1-1+1=1
    \]
    as well. This completes the proof.
\end{proof}

\begin{proof}[Proof of Corollary~\ref{coro:det4}]
    We have the lower bound $\prk(\det_4) \ge 3$ by Remark~\ref{remark:small-prk},
    and the upper bound $\prk(\det_4) \le 3$ from Theorem~\ref{prop:prk-4}.
\end{proof}

\subsubsection{Comparison to other expansions}\label{subsubsec:exp-comp}

For a $4 \times 4$ matrix $A$, the $2$-row Laplace expansion, along rows $I=\{1,2\}$ (say), is
\begin{equation}\label{Laplace-two}
   \det(A) = \sum_J (-1)^{j_1+j_2-1} \det(A_{I,J})\det(A_{I^c,J^c}),
\end{equation}
where the summation is over all $J \in \binom{[4]}{2}$ (all choices of $2$ out of the $4$ columns).\footnote{The coefficients can be written more generality as $(-1)^{\sum_\ell i_\ell+j_\ell}$.}
To formally separate our expansion~(\ref{eq:d4-expansion}) from the Laplace expansion or the generalized Laplace expansion~(\ref{Laplace-two}),
we extend both the preorder of Definition~\ref{reductions} and the notion of equivalence in Definition~\ref{srk-eq} to partition rank in the obvious way. 
Note that this preorder preserves degrees, and so our quadratic expansion~(\ref{eq:d4-expansion}) is certainly not equivalent to the slice rank decomposition given by the Laplace expansion along a single row.

\begin{lemma}\label{not-Laplace}
    The expansions~(\ref{eq:d4-expansion}) and~(\ref{Laplace-two})
    are not equivalent.
\end{lemma}

\begin{proof}
    Let $(\vec{g},\vec{h})$ denote Laplace expansion along any two rows $I$, given by an equation of the form~\eqref{Laplace-two}. 
    This partition rank decomposition of $\det_4$ has $\binom{4}{2}=6$ summands. We claim that the same is true whenever $(\vec{q},\vec{r}) \lesssim (\vec{g},\vec{h})$ is another partition rank decomposition. Establishing this claim will complete the proof.   

    First, note that the quadratics in $\vec{g}$ (i.e. the six $2\times 2$ minors $\det(A_{I,J})$) are linearly independent. Therefore, the decomposition $(\vec{g},\vec{h})$ cannot be shortened by a linear reduction or a syzygy reduction. 
    It thus remains to consider symmetric reductions. 
    We claim that any linear map $T:M_4(\F)\to M_4(\F)$ satisfying $\det_4\circ T = \det_4$ must be an isomorphism. This would imply that the entries of $\vec{g}\circ T$ remain linearly independent quadratics, as needed.
    To prove this claim, suppose for contradiction that there exists a matrix $0\neq A\in M_4(\F)$ satisfying $T(A) = 0.$ Let $C \in M_4(\F)$ be any invertible matrix that shares a row with $A.$ Then
    \[
    0\neq \det_4(C) = \det_4(T(C)) = \det_4 (T(C-A)) = \det_4(C-A) = 0,
    \]
    where the construction of $C$ was used in the first inequality and the last equality, 
    and the linearity of $T$ was used in the second equality.
    This completes the proof.
\end{proof}

As an aside, our expansion~(\ref{eq:d4-expansion}) is quite unremarkable as an arithmetic formula, when
compared to the Laplace expansion or to the two-row Laplace expansion~(\ref{Laplace-two}).
For example, its (leaf) size is $144$ ($24$ variable occurrences in each of the $6$ factors), compared to $64$ for the Laplace expansion, and $48$ for the generalized Laplace expansion.
It also does not improve on its competitors in terms of the number of products, for example. 
Moreover, when applied on random real matrices, it typically produces intermediate reals of  much larger magnitude than the original entries (as each of its factors has as many as $12$ monomials). 
Nevertheless, if we measure these expansions by their partition rank, then~(\ref{eq:d4-expansion}) is strictly better, and indeed best possible.


\section{Separating partition and analytic rank}

In this section we prove an asymptotic separation between partition rank and analytic rank, using Theorem~\ref{main:LB}.
On the other hand, we show that one cannot obtain a better separation---or really any separation at all---by considering random multilinear forms, even when they are conditioned to have high partition rank.

\subsection{Asymptotic separation}

\renewcommand{\T}[2]{\mathcal{T}_{{#1},{#2}}}
\newcommand{\Tr}[3]{\mathcal{T}_{{#1},{#2}}^{(#3)}}

We compute a tight estimate for the probability that $m$ independent, uniformly random vectors in $\F_q^n$ ($\F_q$ is the finite field of size $q$) are linearly dependent.
\begin{lemma}\label{lemma:rank-deficient}
    For any $m \le n$, 
    \[1 \le \Pr_{A \in \F_q^{m \times n}}[\rk(A) < m]/q^{-(n-m+1)} < 1 + 1/(q-1) .\]
\end{lemma}
\begin{proof} 
    The number of $m$-tuples $(v_1,\ldots,v_m)$ of linearly independent vectors in $\F_q^n$ is, as is well known, 
     $\prod_{i=0}^{m-1} (q^n-q^i)$, 
    since each $v_i$ ($i=1,\ldots,m$) can be any vector in $\F_q^n$ that lies outside the subspace spanned by 
    $v_1,\ldots,v_{i-1}$.  
    Thus,
    \[\Pr_{A \in \F_q^{m \times n}}[\rk(A)=m] = q^{-mn}\prod_{i=0}^{m-1} (q^n-q^i) 
    = \prod_{i=0}^{m-1} (1-1/q^{n-i}).\]
    To upper bound this probability, we bound $1-1/q^{n-i} \le 1$ for every $i<m-1$;
    to lower bound it, we use the generalized Bernoulli's inequality $\prod_i (1-x_i) \ge 1-\sum_i x_i$ for $0 \le x_i \le 1$.
    We get
    \[1- \sum_{i=0}^{m-1} 1/q^{n-i}
    \le \Pr_{A \in \F_q^{m \times n}}[\rk(A)=m] \le 1-1/q^{n-m+1}.\]
    Put $k=n-m+1$.
    For the geometric series on the left hand side, we have
    \[\sum_{i=0}^{m-1} 1/q^{n-i} = \sum_{i=k}^{n} 1/q^{j} 
    < q^{-k}/(1-1/q).\]
    Thus, as desired,
    $q^{-k} \le \Pr_{A \in \F_q^{m \times n}}[\rk(A) < m] < (1 + 1/(q-1))q^{-k}$. 
\end{proof}
  

\begin{corollary}\label{coro:ark-det}
    Over any finite field $\F$, and for any $n \ge 2$,
    \[\ceil{\ark(\det_n)} = 2.\]
\end{corollary}
\begin{proof}
    Put $\F=\F_q$.
    For convenience, write $\det_n = \det(A)$ where $A=(a_{i,j})$ is a general $n \times n$ matrix, whose rows are the $n$ variable vectors of $\det_n$.
    The gradient of $\det_n$ with respect to the last row of $A$ is, by the Laplace expansion, 
    \[\nabla \det_n = \big((-1)^{n+1}\det(A_{n,1}),\ldots,(-1)^{n+n}\det(A_{n,n})\big),\]
    where $A_{i,j}$ is obtained from $A$ by removing row $i$ and column $j$.
    Let $A'$ denote the $(n-1)\times n$ matrix obtained from $A$ by removing row $n$.
    Then $\nabla \det_n(A')=0$ if and only if every $(n-1)\times(n-1)$ minor of $A'$ is zero, which is equivalent to $\rk(A') < n-1$.
    In other words,
    \[\Pr[\nabla \det_n = 0] 
    = \Pr_{A' \in \F_q^{(n-1) \times n}}[\rk(A') < n-1].\]
    Lemma~\ref{lemma:rank-deficient}, in the special case $m=n-1$, says
    $q^{-2}
    \le \Pr[\nabla \det_n = 0] 
    < c_q q^{-2}$
    with $c_q=q/(q-1)$.
    Therefore,
    $2 - \log_q c_q < \ark(\det_n) \le 2$.
    Trivially $c_q \le q$,
    so $1 < \ark(\det_n) \le 2$, 
    that is, $\ceil{\ark(\det_n)} = 2$.
%
\end{proof}

\begin{corollary}\label{coro:Cd}
    For any $d \ge 2$,
    \[A(d) \ge \frac{\log_2(d)+1}{2} 
    \xrightarrow[d \to \infty]{} \infty .\]
\end{corollary}
\begin{proof}
    Combining Theorem~\ref{thm:prk-LB} and Corollary~\ref{coro:ark-det}, we get
    \[A(d) 
    = \max_{\text{$d$-linear } T} \frac{\prk(T)}{\ceil{\ark(T)}}
    \ge \frac{\prk(\det_d)}{\ceil{\ark(\det_d)}} \ge \frac{\log_2(d)+1}{2}. \qedhere\]
\end{proof}

\subsection{Random multilinear forms}

Denote by $\T{n}{d}(\F)$ the set of $d$-linear forms $T \colon (\F^n)^d \to \F$.
Let $\Tr{n}{d}{r}(\F)$ be the distribution obtained by summing $r$ reducible forms chosen independently and uniformly at random from $\T{n}{d}(\F)$.
Clearly, this distribution satisfies $\Pr[\prk(T)\le r]=1$, 
and so $\Pr[\ark(T)\le r]=1$.
We will prove that, over any finite field $\F$, the analytic rank of these multilinear forms is roughly $r$ with high probability. Intuitively, this means tensors $T$ with $\prk(T)/\ark(T) >2$ (say) are \emph{rare}, in the sense that only a $o(1)$-fraction of them satisfies this inequality, even when the numerator is fixed.

\begin{remark}
    Other properties of random elements of $\Tr{n}{d}{r}(\F)$ were previously studied in \cite{Bal-Bik, Erman}.  
\end{remark}


We use the following asymptotic notation in $n$: $A \sim B$ if\footnote{Not to be confused with the probabilistic notation of ``distributed as''.} $\lim_{n \to \infty} A/B = 1$ (that is, $A = (1+o(1))B$),
$A \gtrsim B$ if $\lim_{n \to \infty} A/B \ge 1$ (that is, $A \ge (1+o(1))B$),
and $A \lesssim B$ analogously.
\begin{theorem}\label{thm:random-ark}
    Let $\F$ be a finite field, $d \ge 2$, and $\e>0$.
    If $r \le (1-\e)n/2$ then
    \[\E_{T \sim \Tr{d}{n}{r}(\F)}\, \bias(T) \sim |\F|^{-r}.\]
\end{theorem}
\begin{corollary}
    In the conditions of Theorem~\ref{thm:random-ark}, for any $c>0$,
    \[\Pr_{T \sim \Tr{d}{n}{r}(\F)}[\,r-c \le \ark(T) \le r\,] \gtrsim 1-q^{-c}.\]
\end{corollary}
\begin{proof}
    Put $\F=\F_q$.
    By Markov's inequality on the random variable $\bias(T) \ge 0$,
    \[\Pr[\ark(T) \le r - c] = \Pr[\bias(T) \ge q^{-(r-c)}] \le q^{r-c}\E\bias(T) \sim q^{-c}. \qedhere\]
\end{proof}

\begin{remark}
    Our proof works for an arbitrary finite field $\F$.
    As fFor an infinite field $\F$, an essentially identical argument shows that a \emph{generic} $T \in \Tr{d}{n}{r}(\F)$ has geometric rank $r.$  
\end{remark}



We begin with some auxiliary lemmas.
Henceforth, we say that $x \in (\F^n)^d$ is \emph{nontrivial} if $x^{(i)} \neq 0$ for every $i$.
We will also need the following observation about evaluating a random multilinear form.
\begin{lemma}\label{lemma:nontrivial-pt}
    Let $T \in \T{d}{n}$ be chosen uniformly at random.
    For any $x \in (\F^n)^d$ nontrivial,
    $T(x)$ is distributed uniformly in $\F$.
\end{lemma}
\begin{proof}
    We have
    $T(x^{1},\ldots,x^{(d)}) = \sum_{i_1,\ldots,i_d} c_{i_1,\ldots,i_d} x^{(1)}_{i_1}\cdots x^{(d)}_{i_d}$ with $c_{i_1,\ldots,i_d} \in \F$ chosen uniformly at random.
    Abbreviate $X_{i_1,\ldots,i_d} = x^{(1)}_{i_1}\cdots x^{(d)}_{i_d}$.
    Since $x$ is nontrivial, for every $1 \le i \le d$ there is $1 \le j_i \le n$ such that $x^{i}_{j_i} \neq 0$. Put $J = (j_1,\ldots,j_d)$.
    We deduce that $T(x) \in \F$ is distributed uniformly since $c_J X_J$ is: 
    for every $y \in \F$, we have $T(x)=y$ if and only if 
    $c_J = (y - \sum_{I \neq J} c_IX_I)/X_J$.
\end{proof}

We will later use the following elementary observations to bound the effects of low probability (``bad'') events.
\begin{lemma}\label{lemma:diff}
    Let $A,B,E$ be finite sets.
    If $|A \sm E| = |B \sm E|$
    then $\Big||A|-|B|\Big| \le |E|$.
\end{lemma}
\begin{proof}
    We have
    $|A| = |A \sm E| + |A \cap E| = |B \sm E| + |A \cap E|$
    by assumption.
    This implies the bound 
    $|A| \ge |B \sm E| \ge |B| - |E|$,
    and the bound
    $|A| \le |B| + |E|$.
\end{proof}

\begin{lemma}\label{lemma:diff-1}
    For events $A$ and $E$, $\big|\Pr[A]-\Pr[A \vert E^c]\big| \le \Pr[E]$.
\end{lemma}
\begin{proof}
    We prove the identity
    $\Pr[A]-\Pr[A \vert E^c] = \Pr[A \cap E] - \Pr[E]\Pr[A|E^c]$,
    from which the desired bounds follow since $0 \le \Pr[A \cap E] \le \Pr[E]$.
    By the law of total probability,
    $\Pr[A] - \Pr[A \cap E^c] = \Pr[A \cap E]$,
    and $\Pr[A \cap E^c] = \Pr[A \vert E^c](1-\Pr[E])$,
    so the desired identity follows.
\end{proof}

    
    

    

We are now ready to give the proof.

\begin{proof}[Proof of Theorem~\ref{thm:random-ark}]
    Put $V=\F^n$.
    Henceforth, for $T \sim \Tr{d}{n}{r}(\F)$ 
    with a given partition-rank decomposition $T=\sum_{i=1}^r R_iS_i$, we let $R_i$, for concreteness, depend on $x^{(d)}$ (so $S_i$ does not).
    For $T \sim \T{d}{n}$ and $x=(x^{(1)},\ldots,x^{(d-1)}) \in V^{d-1}$, 
    we define the events
    \[A = \{(x,T) \vert \nabla T(x)=0\},\]
    \[B = \{(x,T) \vert S_1(x)=\cdots=S_r(x)=0\} ,\]
    \[E = \{(x,T) \vert \nabla R_1(x),\ldots,\nabla R_r(x) \text{ are linearly dependent}\} .\]
    We claim that $A \sm E = B \sm E$.
    Indeed, observe that 
    $\nabla T(x) = \sum_{i=1}^r S_i(x)\nabla R_i(x)$.
    Thus, for $x$ fixed, $\nabla T(x) \in \F^n$ can be written as a linear combination of $\nabla R_1(x),\ldots,\nabla R_r(x) \in \F^n$ with respective coefficients $S_1(x),\ldots,S_r(x) \in \F$.
    Therefore, 
    if $(x,T) \notin E$ then we have the equivalence
    $(x,T) \in A$ if and only if $(x,T) \in B$. That is, $A \sm E = B\sm E$, as claimed.
    Therefore, we may apply Lemma~\ref{lemma:diff} to deduce that
    \begin{equation}\label{eq:pr-dev}
        \Big|\Pr[A]-\Pr[B]\Big| \le \Pr[E] .
    \end{equation}
    Note that 
    $\Pr[A] = \E_T \Pr_x[A \vert T] = \E_T \bias(T)$.
    We will prove that 
    \begin{equation}\label{eq:random-ark-goal1}
        \Pr[B] \sim q^{-r},
    \end{equation}
    and that
    \begin{equation}\label{eq:random-ark-goal2}
        \Pr[E] \le 3q^{-(n-r+1)}.
    \end{equation}
    To see why these suffice, 
    observe that~(\ref{eq:pr-dev}) implies, assuming~(\ref{eq:random-ark-goal2}), that
    \[\Big|\E_T \bias(T)/q^{-r}-\Pr[B]/q^{-r}\Big| \le 3q^{-(n-2r+1)} 
    \le 3q^{-(r+\e n)},\]
    where the last inequality uses the statement's assumption $r \le (1-\e)n/2$.
    Assuming~(\ref{eq:random-ark-goal1}), and taking the limit as $n\to\infty$, we get $\E_T \bias(T) \sim q^{-r}$, as desired. 

    It therefore remains to prove~(\ref{eq:random-ark-goal1}) and~(\ref{eq:random-ark-goal2}). We will use the event $F = \{(x,T) \vert x \text{ is trivial}\}$.
    For~(\ref{eq:random-ark-goal1}), recall that the coefficients of $S_1,\ldots,S_r$ are chosen independently and uniformly at random. 
    By Lemma~\ref{lemma:nontrivial-pt}, for any nontrivial $x \in V^d$
    we have that $S_1(x),\ldots,S_r(x)$ are $r$ independent, uniformly random scalars in $\F$.
    Therefore, $\Pr[B \vert F^c] = q^{-r}$.
    By Lemma~\ref{lemma:diff-1},
    $\big|\Pr[B] - \Pr[B \vert F^c]\big| \le \Pr[F] = q^{-(d-1)n}$, which thus proves~(\ref{eq:random-ark-goal1}).
    
    For~(\ref{eq:random-ark-goal2}), observe that, since the coefficients of $R_1,\ldots,R_r$ are chosen independently and uniformly at random (from $\F$), 
    so are the coefficients of $\nabla R_1,\ldots,\nabla R_r$.
    By Lemma~\ref{lemma:nontrivial-pt}, for any nontrivial $x \in V^{d}$
    we have that $\nabla R_1(x),\ldots,\nabla R_r(x)$ are $r$ independent, uniformly random vectors in $\F^n$.
    Therefore, by Lemma~\ref{lemma:rank-deficient}, $\Pr[E \vert F^c] \le 2q^{-(n-r+1)}$.
    By Lemma~\ref{lemma:diff-1},
    $\big|\Pr[E] - \Pr[E \vert F^c]\big| \le \Pr[F] = q^{-(d-1)n}$,
    which thus proves~(\ref{eq:random-ark-goal2}).
    As explained above, this completes the proof.
\end{proof}

\end{document}